\documentclass[12pt, a4paper, reqno]{amsart}
\usepackage{amscd, amssymb}
\usepackage{amsthm}
\usepackage{tikz}
\usetikzlibrary{arrows,positioning} 
\usepackage{mathrsfs}
\usepackage{bbm}
\usepackage{ stmaryrd }

\usepackage{hyperref}
\hypersetup{
	colorlinks=true,
	linkcolor=blue,
	filecolor=magenta,  
	urlcolor=cyan,
}

\usepackage{hyperref}
\usepackage{cleveref}

\def\End{\operatorname{End}}

\def\dim{\operatorname{dim}}

\def\Dist{\operatorname{Dist}}

\def\Ind{\operatorname{Ind}}
\def\Coind{\operatorname{Coind}}

\def\Hom{\operatorname{Hom}}

\def\Rep{\operatorname{Rep}}

\def\UU{\mathcal{U}}

\def\p{\mathfrak{p}}

\def\g{\mathfrak{g}}

\def\k{\mathfrak{k}}
\def\l{\mathfrak{l}}
\def\s{\mathfrak{s}}
\def\o{\mathfrak{o}}

\def\ol{\overline}

\def\sub{\subseteq}

\def\xto{\xrightarrow}

\newtheorem{thm}{Theorem}[section]

\newtheorem{lemma}[thm]{Lemma}
\newtheorem{prop}[thm]{Proposition}

\theoremstyle{definition}

\theoremstyle{remark}
\newtheorem{remark}[thm]{Remark}

\numberwithin{equation}{section}

\usepackage{relsize}
\usepackage{fancyhdr}
\usepackage[all,cmtip]{xy}

\begin{document}
\title{Two geometric proofs of the classification of algebraic supergroups with semisimple representation theory}

\author[Alex Sherman]{Alexander Sherman}

\maketitle
\pagestyle{plain}

\begin{abstract}  We present two novel proofs of the known classification of connected affine algebraic supergroups $G$ such that $\Rep G$ is semisimple.  The proofs are geometrically motivated, although both rely on an algebraic lemma that characterizes $\o\s\p(1|2n)$ amongst simple Lie superalgebras.  
\end{abstract}

\section{Introduction}

Let $G$ be a connected affine algebraic supergroup over an algebraically closed field $k$ of characteristic zero.  When is the category $\Rep(G)$ semisimple? 

If $G=G_0$ is an algebraic group, the answer to this question is elegant: $\Rep(G)$ is semisimple if and only if $G$ is reductive, meaning that $R_u(G)$, the unipotent radical of $G$, is trivial.  Equivalently, $G$ has a faithful semisimple representation.  This theorem leads to the study of reductive algebraic groups, which has a well-known and beautiful classification in terms of root datum.  

For supergroups, semisimplicity is a rare phenomenon.  Perhaps the most natural supergroup to study is the general linear supergroup, $G=GL(V)$, for a super vector space $V$.  However its representation category is not semisimple unless $V$ is purely even or purely odd, i.e. unless $G=G_0$. 

In \cite{hochschild1976semisimplicity} and \cite{djokovic1976semisimplicity}, it was proven that the only simple Lie superalgebra which  has semisimple finite-dimensional representation theory and is not purely even is $\o\s\p(1|2n)$.  If $\g$ is an arbitrary Lie superalgebra with semisimple finite-dimensional representation theory, then it must be a direct product of simple ones with the same property, so this effectively answers our question on the level of Lie superalgebras.  We write the formal statement on the level of supergroups in \cref{main_thm_intro}.  

However the original proof of Djokovic and Hochschild came before Kac's seminal paper on Lie superalgebras (\cite{kac1977lie}) and before much of the existing technology in the study of Lie superalgebras and supergroups had been developed.  We seek to supplement their original result with new proofs which also provide insight into the lack of semisimplicity for the representation theory of Lie superalgebras and supergroups.

Therefore we present two geometric proofs of the following classification theorem:
\begin{thm}\label{main_thm_intro}
	Let $G$ be a connected algebraic supergroup.  Then $\Rep(G)$ is semisimple if and only if 
	\[
	G\cong K\times OSp(1|2n_1)\times\cdots\times OSp(1|2n_k)
	\]
	for a reductive algebraic group $K$ and positive integers $n_1,\dots,n_k$.  
\end{thm}
We follow the convention of \cite{entova2020jacobson} to write $OSp(1|2n)$ for what is more precisely written $SOSp(1|2n)$, i.e. the form-preserving automorphisms of $(k^{1|2n},(-,-))$ with Berezinian 1.  

The statement of the theorem is reduced, via an algebraic argument, to the following statement.  Define
\[
\g_{\ol{1}}^{ss}=\{u\in\g_{\ol{1}}:u^2(=[u,u]/2)\text{ is semisimple in }\g_{\ol{0}}\}.
\]
Note that this space was considered previously by Serganova (in private communication).
\begin{thm}\label{main_redn_thm}
	$\Rep G$ is semisimple if and ony if $G_0$ is reductive and $\g_{\ol{1}}^{ss}=\{0\}$.  
\end{thm}
We note the following nice interpretation of \cref{main_redn_thm}.  If $u\in\g_{\ol{1}}^{ss}$, write $h:=u^2$. We define $DS_u$ to be the functor $V\mapsto (V^h)_u$ on the category of $\g$-modules.  Here, $V^h$ denotes the $h$-invariants of $V$, and since $u^2=h$ we have that $u$ acts by a square-zero operator on $V^h$, and $(V^h)_u$ denotes its cohomology.  We call $DS_u$ a Duflo-Serganova functor.

Then $DS_u$ is a tensor functor and enjoys many of the same properties as the usual Duflo-Serganova functor defined in \cite{duflo2005associated}, where $u^2=0$ by assumption.  We may consider $\g_{\ol{1}}^{ss}$ as the parameter space for Duflo-Serganova functors on the category of $\g$-modules.  Then \cref{main_redn_thm} tells us that the existence of Duflo-Serganova functors is only a feature of algebraic supergroups with non-semisimple representation theory.

We also call attention to recent work of Entova and Serganova in \cite{entova2020jacobson} where a different equivalent condition for semisimplicity of $\Rep G$ was given.  There, the authors define a subset $\g_{neat}\sub\g_{\ol{1}}$, which consists of all $u\in\g_{\ol{1}}$ with the property that $u^2$ is nilpotent, and such that the action of $u$ on any finite-dimensional representation of $G$ may be extended to an action of $\o\s\p(1|2)$.  They prove that $\Rep(G)$ is semisimple if and only if $\g_{neat}=\g_{\ol{1}}$.  One consequence of this is that if $\Rep(G)$ is semisimple, then $\Hom(\mathbb{G}_{a}^{1|1},G)=\g_{\ol{1}}$, where $\mathbb{G}_a^{1|1}$ is the supergroup with Lie superalgebra generated freely by one odd element and $(\mathbb{G}_a^{1|1})_0=\mathbb{G}_a$.  

On the other hand, \cref{main_redn_thm} may be stated by saying that $\Hom(\mathbb{M}/\mathbb{G}_a,G)=\{0\}$, where $\mathbb{M}$ is the minuscule supergroup defined in section 5 of \cite{entova2020jacobson}.  The supergroup $\mathbb{M}/\mathbb{G}_a$ has a one-dimensional odd part and an infinite-dimensional pro-reductive central even part, i$.$e$.$ an infinite-dimensional torus.  We may view $\Rep(\mathbb{M}/\mathbb{G}_a)$ as the category of representations of the two-dimensional Lie superalgebra freely generated by an odd element $u$ such that $[u,u]$ acts semisimply.  In general we have an identification $\Hom(\mathbb{M}/\mathbb{G}_a,G)=\g_{\ol{1}}^{ss}$.

\begin{remark}
	The author has declined to use a term for the property of a supergroup that $\Rep(G)$ be semisimple.  This is due to fear of overuse of particular words, e.g. semisimple or reductive.  To avoid confusion, we don't give this property a name.
\end{remark}

\subsection{Acknowledgments}  The author would like to thank Vera Serganova, Alexander Alldridge, and Inna Entova-Aizenbud for helpful discussions.  This research was partially supported by ISF grant 711/18 and NSF-BSF grant 2019694.

\section{Preliminary arguments}

We refer to \cite{carmeli2011mathematical} for background on algebraic supergroups.  We write $G$ for a connected affine algebraic supergroup, and $\g$ for its Lie superalgebra. If $u\in\g_{\ol{1}}$ we will write $u^2$ for $[u,u]/2$. The even part of $G$ is written $G_0$.  Let $\Rep(G)$ denote the category of rational $G$-modules.  Recall that $\Rep(G)$ is equivalent to the category of $\g$-modules that also integrate to an action of $G_0$.  

We write $k$ for the trivial representation of $G$.  The following is well-known.

\begin{lemma}\label{proj_k}
	The category $\Rep(G)$ is semisimple if and only if $k$ is projective.   
\end{lemma}

%

\subsection{Coordinate ring of $G$} Write $k[G]$ for the coordinate ring of functions on $G$, a supercommutative Hopf superalgebra.  We will view $k[G]$ as a $G$-module according the action by left translation, or as a $G\times G$-module according to the actions by left and right translation.  Recall that as a $G$-module we have $k[G]\cong\Ind_{G_0}^{G}k[G_0]$, i$.$e$.$ we have an identification of superalgebras (see \cite{koszul1982graded})
\[
k[G]=\Hom_{\g_{\ol{0}}}(\UU\g,k[G_0]).
\]
Further, this identification is $G_0$-equivariant, in the senses that it realizes $k[G]$ as $S(\g_{\ol{1}}^*)\otimes k[G_0]$ as a $G_0$-module.  In particular, $k[G_0]$ is a $G_0$-stable subalgebra of $k[G]$.  

\begin{lemma}\label{k_splits_k[G]_lemma}
	The following are equivalent:
	\begin{enumerate}
		\item The $G$-submodule $k\cdot 1$ splits off of $k[G]$.
		\item The $G\times G$ submodule $k\cdot 1$ splits off of $k[G]$.
		\item $k$ is projective as a $G$-module.  
	\end{enumerate}
\end{lemma}
\begin{proof}
Recall that a $G$-module is the same as a $k[G]$-comodule.  In the category of $k[G]$-comodules, $k[G]$ is injective, and thus if $k\cdot 1$ splits off it must be injective as well.  By duality, $k$ will also be projective.  Now $(1)\iff(3)$ follow and from this it is easy to prove $(2)\iff(3)$. 
\end{proof}

\begin{lemma}
	If $\Rep(G)$ is semisimple, then $G_0$ is reductive.
\end{lemma}
\begin{proof}
	By \cref{k_splits_k[G]_lemma}, if $\Rep(G)$ is semisimple then $k\cdot 1$ splits off $k[G]$ as a $G$-module.  However $k[G_0]$ may be realized as a $G_0$-subalgebra of $k[G]$, and thus $k\cdot 1$ splits off of $k[G_0]$.  By \cref{k_splits_k[G]_lemma} we have $k$ is a projective $G_0$-module and thus $\Rep(G_0)$ is semisimple, and therefore $G_0$ is reductive.  
\end{proof}
A supergroup $G$ is called quasireductive if $G_0$ is reductive.  We have thus shown that if $\Rep(G)$ is semisimple, then $G$ is quasireductive.  Quasireductive supergroups and in particular their representation theory were studied in \cite{serganova2011quasireductive} amongst other places.

\subsection{Properties of $\g$}
Recall that a Lie superalgebra $\g$ is quasireductive if $\g_{\ol{0}}$ is reductive and $\g_{\ol{1}}$ is a semisimple $\g_{\ol{0}}$-module.  In particular the Lie superalgebra of a quasireductive Lie supergroup is itself quasireductive.
\begin{lemma}\label{product_of_simples}
	If $\Rep(G)$ is semisimple, then $\g=\operatorname{Lie} G$ is a product of its center and simple quasireductive Lie superalgebras.  Further, its center is purely even.
\end{lemma}
\begin{proof}
	The first statement follows immediately from semisimplicity and the reductivity of $G_0$.  For the second statement, if the center had a non-trivial odd component then $G$ would admit $\mathbb{G}_{a}^{0|1}$ as a quotient, whose representation theory is not semisimple.
\end{proof}

\subsection{The Lie supergroup $OSp(1|2n)$}\label{osp(1|2n)_subsec}

We define the Lie supergroup $OSp(1|2n)$ in terms of a corresponding super Harish-Chandra pair.  First, let $\o\s\p(1|2n)$ be the Lie superalgebra with $\o\s\p(1|2n)_{\ol{0}}=\s\p(2n)$ and $\o\s\p(1|2n)_{\ol{1}}=k^{2n}$, the standard representation of $\s\p(2n)$.  Write $(-,-)$ for an $\s\p(2n)$-invariant symplectic form on $k^{2n}$. Then the odd Lie bracket $[-,-]:S^2k^{2n}\to\s\p(2n)$ is given by
\[
[u,v](w)=(u,w)v+(v,w)u
\]
where we view $\s\p(2n)$ as a subspace of $\g\l(k^{2n})$.  Notice that in this way $[-,-]:S^2k^{2n}\to\s\p(2n)$ defines an isomorphism, and that any other $\s\p(2n)$-equivariant map between these spaces differs from it by a scalar.  

Now consider the super Harish-Chandra pair $(Sp(2n),\o\s\p(1|2n))$, where $Sp(2n)$ has the natural action on $\o\s\p(1|2n)$.  This forms a super Harish-Chandra pair, and thus defines a (quasireductive) Lie supergroup, which we call $OSp(1|2n)$.  As stated in the introduction, this supergroup is perhaps better denoted $SOSP(1|2n)$, but we drop the first $S$ for aesthetic purposes.

\begin{lemma}\label{osp_no_covers_etc}
	$OSp(1|2n)$ is the unique connected supergroup $G$ with $\operatorname{Lie} G\cong\o\s\p(1|2n)$.  Thus, for positive integers $n_1,\dots,n_k$, $OSp(1|2n_1)\times\cdots\times OSp(1|2n_k)$ is the unique connected supergroup with Lie superalgebra $\o\s\p(1|2n_1)\times\cdots\times\o\s\p(1|2n_k)$.  
\end{lemma}
\begin{proof}
Suppose $G$ is a connected supergroup with $Lie G\cong\o\s\p(1|2n)$.  Then since $Sp(2n)$ is simply connected, $G_0$ must be a quotient of $Sp(2n)$ by a finite central subgroup.  However, the standard representation $k^{2n}$ of $Sp(2n)$ is faithful, and must descend to a representation of $G_0$.  Thus we are forced to have $G_0\cong Sp(2n)$.  The statement now follows.
\end{proof}

\subsection{A sufficient criteria}
Define
\[
\g_{\ol{1}}^{ss}=\{u\in\g_{\ol{1}}:u^2\text{ is semisimple in }\g_{\ol{0}}\}.
\]
The following proof uses the ideas of Proposition 3.2 and Theorem 4.1 of \cite{djokovic1976semisimplicity}.  Note that one could easily use the classification of simple Lie superalgebras, but we use a simpler, more direct proof.
\begin{prop}\label{sufficient_osp(1|2n)}
	Suppose that $\g$ is a simple quasireductive Lie superalgebra with $\g_{\ol{1}}\neq0$, such that $\g_{\ol{1}}^{ss}=\{0\}$.  Then $\g\cong\o\s\p(1|2n)$.
\end{prop}
\begin{proof}
	Choose a Cartan subalgebra $\mathfrak{t}\sub\g_{\ol{0}}$, which gives rise to a root decomposition on $\g$.  Let $u\in\g_{\ol{1}}$ be a non-zero root vector of weight $\alpha$.  Note that $\alpha\neq0$, else $u^2\in\mathfrak{t}$, a contradiction.  Then $u^2$ is of weight $2\alpha$, and must be non-zero of course, else $u^2$ is semisimple.  It follows that $2\alpha$ is a root of $\g_{\ol{0}}$ and is divisible by $2$ as a root.  By the classification of irreducible root systems, $2\alpha$ must belong to a root system of type $C_n$, and $2\alpha$ is a long root.  Further, if $\dim\g_{\alpha}>1$ then the map $[-,-]:S^2\g_{\alpha}\to\g_{2\alpha}$ must give rise to some non-zero $v\in\g_{\alpha}$ such that $v^2=0$, since $\dim\g_{2\alpha}=1$.  Thus by our assumption, $\dim\g_{\alpha}=1$.
  
  So far we have shown that $\g_{\ol{1}}$ is a multiplicity-free sum of the standard module of different copies of $\s\p$ in $\g_{\ol{0}}$.  Let $V_1,V_2$ be distinct irreducible summands in $\g_{\ol{1}}$.  Then since $[-,-]:V_1\otimes V_2\to\g_{\ol{0}}$ is $\g_{\ol{0}}$-equivariant, it must be 0, and thus $V_1$ and $V_2$ commute.  Let $V\sub\g_{\ol{1}}$ be an irreducible summand, and let $\s\p(V)\sub\g_{\ol{0}}$ be the Lie superalgebra acting on it non-trivially.  Let $\k=\s\p(V)\oplus V$.  Then from what we have shown, it follows that $\k$ is an ideal of $\g$.  Therefore $\k=\g$.
  
  Now $S^2V\cong\s\p(V)$, and thus either $[-,-]:S^2V\to\s\p(V)$ is an isomorphism or is the zero map.  If it is the zero map, then $\g_{\ol{1}}=V$ defines an abelian ideal of $\g$, a contradiction, so instead this map must be an isomorphism.  It is now easy to prove that $\g\cong\o\s\p(1|2n)$.  
\end{proof}

\begin{thm}\label{thm_for_redn}
	Suppose that $G$ is a quasireductive Lie supergroup such that $\Rep(G)$ is semisimple, and $\g_{\ol{1}}^{ss}=\{0\}$.  Then there exists a reductive even subgroup $K$ of $G$, and positive integers $n_1,\dots,n_k$ such that
	\[
	G\cong K\times OSp(1|2n_1)\times\cdots\times OSp(1|2n_k).
	\]
\end{thm}
\begin{proof}
	By \cref{product_of_simples}, we may write $\g=\k\times\g_1\times\cdots\times \g_k$ for a reductive Lie algebra $\k$ and simple quasireductive Lie superalgebras $\g_1,\dots,\g_k$ with nonzero odd part.  By \cref{sufficient_osp(1|2n)}, $\g_i\cong \o\s\p(1|2n_i)$ for some positive integer $n_i$.  Let $\k$ integrate to the connected subgroup $K$ of $G$, and $\g_1\times\cdots\times\g_k$ integrate to the connected subgroup $OSp(1|2n_1)\times\cdots\times OSp(1|2n_k)$ of $G$, where we use \cref{osp_no_covers_etc}.  Thus we have a map
	\[
	K\times OSp(1|2n_1)\times\cdots\times OSp(1|2n_k)\to G
	\]
	which induces an isomorphism of Lie superalgebras. Further, $K_0\cap(OSp(1|2n_1)_0\times\cdots\times OSp(1|2n_k)_0)$ is normal in $OSp(1|2n_1)\times\cdots\times OSp(1|2n_k)$, and so must be trivial by \cref{osp_no_covers_etc}.  Hence we find that
	\[
	K\times OSp(1|2n_1)\times\cdots\times OSp(1|2n_k)\cong G.
	\]
\end{proof}

With \cref{thm_for_redn} proven, it remains to prove the following.

\begin{thm}\label{main_reduced_statement}
	If $\Rep(G)$ is semisimple, then $\g_{\ol{1}}^{ss}=\{0\}$.
\end{thm}

For this we offer two proofs, along with a proof that $\Rep(OSp(1|2n))$ is indeed semisimple.

\section{Proof I: vector fields on $G$}

In light of \cref{k_splits_k[G]_lemma}, in order to prove \cref{main_reduced_statement}, we show the following:
\begin{prop}\label{no_splitting_ss_odd}
	Suppose that $\g_{\ol{1}}^{ss}\neq\{0\}$.  Then $k\cdot 1$ does not split off $k[G]$.
\end{prop}

For the proof, we need an important lemma.

\begin{lemma}\label{odd_ss_X_no_splitting_lemma}
	Let $X$ be an affine supervariety, and suppose that $u$ is an odd vector field on $X$ such that $u^2$ acts semisimply on $k[X]$.  Suppose further that $u$ is everywhere non-vanishing as a vector field.  Then there exists a function $f\in A$ such that $uf=1$.  In particular, $k\cdot 1$ does not split off $k[X]$ as a module over the Lie superalgebra generated by $u$. 
\end{lemma}
\begin{proof}
	Since $u$ is everywhere non-vanishing and $X$ is affine, for each closed point $x\in X(k)$ there exist an odd function $\xi_x\in k[X]$ such that $(u\xi_x)(x)\neq0$.  Thus the ideal generated by $\{u\xi_x\}_{x\in X(k)}$ is equal to $k[X]$, so we may choose finitely many functions in this set, say $u\xi_1,\dots,u\xi_k$, and even functions $g_1,\dots,g_k$ such that
	\[
	g_1u\xi_1+\dots+g_ku\xi_k=1.
	\] 
Set $p=g_1\xi_1+\dots+g_k\xi_k$.  We compute that
	\[
	up=g_1u\xi_1+\dots+g_ku\xi_k+u(g_1)\xi_1+\dots+u(g_k)\xi_k=1+\eta
	\]
	where $\eta=u(g_1)\xi_1+\dots+u(g_k)\xi_k$.  Since $u(g_i)$ is odd, $\eta\in(k[X]_{\ol{1}})^2$, and this nilpotent.  Write $h=u^2$ and recall that $[h,u]=0$ by the Jacobi identity.  We may then write $p=\sum\limits_{\lambda}p_{\lambda}$ and $\eta=\sum\limits_{\lambda}\eta_{\lambda}$ according the decomposition of $k[X]$ into eigenspaces of $h$.  Since $h$ preserves $(k[X]_{\ol{1}})^2$, $\eta_{\lambda}\in(k[X]_{\ol{1}})^2$ for all $\lambda$.  Now we see that
	\[
	up=\sum\limits_{\lambda}u(p_{\lambda})=1+\sum\limits_{\lambda}\eta_{\lambda}.
	\]
	Since $h(1)=0$, we find that $\alpha:=u(p_{0})=1+\eta_0$ is a unit in $k[X]$.  Further, $u\alpha=u^2(fp_0)=h(p_0)=0$, so if we set $f:=p_0/\alpha$, we find that
	\[
	u(f)=u(p_0/\alpha)=\alpha/\alpha=1,
	\]
	and we are done.
\end{proof}

We are now ready to prove \cref{no_splitting_ss_odd}.

\begin{proof}
	Suppose that there exists a non-zero element $u$ of $\g_{\ol{1}}^{ss}$. Consider the right-invariant vector field $u_{R}$ on $G$ determined by $u$.  By construction $u_R$ is non-vanishing and $(u_R)^2=(u^2)_{R}$ acts semisimply on $k[G]$.  Then by \cref{odd_ss_X_no_splitting_lemma} there exists a function $f\in k[G]$ such that $uf=1$.  Hence $k\cdot 1$ does not split off $k[G]$, so $\Rep(G)$ is not semisimple.
\end{proof}

\section{Proof II: a differential operator on $G/G_0$}

Consider the homogeneous $G$-space $G/G_0$.  As a $G_0$-module, and as an algebra, we have an isomorphism
\[
k[G/G_0]\cong \Lambda^\bullet\g_{\ol{1}}^*.
\]
In particular $G/G_0$ has one point as a space, and a finite-dimensional algebra of functions.  As a $G$-module we have an identification
\[
	k[G/G_0]=\Ind_{G_0}^{G}k=\Coind_{\g_{\ol{0}}}^{\g}k.
\]
Since $G_0$ is reductive, $k[G/G_0]$ is injective, and explicitly we may write
\[
k[G/G_0]\cong\bigoplus\limits_{L}I(L)^{\oplus [\operatorname{Res}_{G_0}^{G}L:k]},
\]
where $I(L)$ denotes the indecomposable injective $G$-module with socle $L$, an irreducible $G$-module.  In particular, $I(k)$ appears with multiplicity one, so we may write
\[
k[G/G_0]=I(k)\oplus M
\]
for some $G$-module $M$.  Now $\Rep G$ is semisimple if and only if $I(k)=k=P(k)$. Thus assume from now on that $I(k)=P(k)$.  In this case, there exists a operator $\phi\in\End(I(k))$ which takes the head to the tail.  Define $\Phi:=\phi\oplus\mathbf{0}_{M}\in\End(k[G/G_0])$.  The following are straightforward to prove.
\begin{lemma}\label{properties_of_Phi}
	\begin{enumerate}
		\item $\Phi\circ u=u\circ \Phi=0$ for all $u\in\g$;
		\item $\Phi^2\neq 0$ if and only if $I(k)=k=P(k)$ if and only if $\Rep(G)$.  
	\end{enumerate}
\end{lemma}

Since $k[G/G_0]$ is finite-dimensional, we have that $\End(k[G/G_0])=D(G/G_0)$, where $D(G/G_0)$ is the algebra of differential operators on $G/G_0$.  Thus $\Phi$ is a differential operator, and in fact is a $G$-equivariant differential operator by construction.  Write $D^{G}(G/G_0)$ for the space of $G$-equivariant differential operators on $G/G_0$.  Then we have the well known isomorphism
\[
\operatorname{res}:D^{G}(G/G_0)\xto{\sim}\Dist(G/G_0,eG_0)^{G_0}\cong(\UU\g/\g_{\ol{0}}\UU\g)^{G_0}.
\]
Here $\Dist(G/G_0,eG_0)$ denotes the distributions on $G/G_0$ supported at the only point, so in fact this is nothing but $k[G/G_0]^*$ as a $G$-module.   
\begin{lemma}
	$k\cdot \operatorname{res}(\Phi)=(\UU\g/\g_{\ol{0}}\UU\g)^{G}$.
\end{lemma}

\begin{proof}
	We have an identification
	\[
	\UU\g/\g_{\ol{0}}\UU\g\cong\Ind_{\g_{\ol{0}}}^{\g}k,
	\]
	so $\dim (\UU\g/\g_{\ol{0}}\UU\g)^{G}\leq 1$.  On the other hand \cref{properties_of_Phi} implies that $\operatorname{res}(\Phi)\in(\UU\g/\g_{\ol{0}})^{G}$, so we are done.
\end{proof}

We write $v:=\operatorname{res}(\Phi)$.  

\begin{remark}
	The element $v$ is an example of a ghost distribution on the symmetric space $G/G_0$.  Ghost distributions were defined by the author in \cite{sherman2021ghost} for any supersymmetric space.  
\end{remark}

\begin{remark}
	Note that $v$ agrees (up to scalar) with the element $v_{\emptyset}$ defined in section 3 of \cite{gorelik2000ghost}.  This is not difficult to show given what we have proven thus far, and we refer \cite{sherman2021ghost} for a full proof. 
\end{remark}
The next lemma is straightforward.
\begin{lemma}
	We have 
	\[
	\varepsilon(v)=\Phi(1)(eG_0),
	\]
	where $\varepsilon$ is the counit on $\UU\g$.  Thus $\Rep G$ is semisimple if and only if $\varepsilon(v)\neq0$.
\end{lemma}

We now may prove \cref{main_reduced_statement}.

\begin{proof}
	Suppose that $u\in\g_{\ol{1}}^{ss}$ is nonzero.  Then $DS_{u}\UU\g/\g_{\ol{0}}\UU\g=0$ because $\UU\g/\g_{\ol{0}}\UU\g$ is projective, being induced from $\g_{\ol{0}}$.  However $u\cdot v=0$ by definition of $v$, so there exists $v'\in\UU\g/\g_{\ol{0}}\UU\g$ such that $v=u\cdot v'=v'u$.  But then
	\[
	\varepsilon(v)=\varepsilon(v'u)=\varepsilon(v')\varepsilon(u)=0.
	\]
\end{proof}

\subsection{Reductivity of $OSp(1|2n)$}

All that remains is to prove that $\Rep OSp(1|2n)$ is semisimple.  Let $\g=\o\s\p(1|2n)$.  By our work above, it suffices to prove that for a non-zero element $v\in(\UU\g/\g_{\ol{0}}\UU\g)^{\g}$ we have $\varepsilon(v)\neq0$.  To show this we only need to construct $v$ and check that $\varepsilon(v)\neq0$.  

However this has already been done in \cite{djokovic1976semisimplicity}.  There, they prove that the element
\[
v=(1+t_1)(3+t_2)\cdots((2n-1)+t_n)
\]
is $\g$-invariant, where $t_i=a_{i}b_i$, and $(a_1,\dots,a_n,b_1,\dots,b_n)$ are a basis of $\g_{\ol{1}}$ such that $(a_i,b_j)=\delta_{ij}$, where $(-,-)$ is $\s\p(2n)$-invariant form on $\g_{\ol{1}}$ as discussed in \cref{osp(1|2n)_subsec}.  Note that in \cite{djokovic1976semisimplicity} they actually consider the $\g$-invariants of $\UU\g/(\UU\g)\g_{\ol{0}}$.  To go between their case and ours we only need apply the antipode on $\UU\g$.

\bibliographystyle{amsalpha}
\bibliography{bibliography}

\textsc{\footnotesize Dept. of Mathematics, Ben Gurion University, Beer-Sheva,	Israel} 

\textit{\footnotesize Email address:} \texttt{\footnotesize xandersherm@gmail.com}

\end{document}